\documentclass{amsart}

\newtheorem{theorem}{Theorem}[section]
\newtheorem{lemma}[theorem]{Lemma}
\newtheorem{prop}[theorem]{Proposition}

\theoremstyle{definition}

\theoremstyle{remark}
\newtheorem{remark}[theorem]{Remark}

\numberwithin{equation}{section}
\numberwithin{table}{section}

\usepackage{amssymb,latexsym,amsfonts}

\newcommand{\li}{\mbox{Li}}
\newcommand{\leqs}{\leqslant}
\newcommand{\geqs}{\geqslant}
\newcommand{\eps}{\varepsilon}
\renewcommand{\theta}{\vartheta}

\begin{document}
\title{Estimates of some functions over primes \\ without R.H.}

\author{Pierre DUSART}
\address{XLIM UMR 6172, Facult\'e des Sciences et Techniques, Universit\'e de Limoges, 87060, FRANCE}
\curraddr{D\'epartement de Math\'ematiques,
123 avenue Albert THOMAS, 87060 LIMOGES CEDEX, FRANCE}
\email{pierre.dusart@unilim.fr}


\date{January 24, 2007 
.}



\begin{abstract}
Some computations made about the Riemann Hypothesis and in particular, the verification that zeroes of $\zeta$ belong on the critical line and the extension of zero-free region are useful to get better effective estimates of number theory classical functions which are closely linked to $\zeta$ zeroes like $\psi(x),\theta(x),\pi(x)$ or the $k^{th}$ prime number $p_k$.
\end{abstract}

\maketitle

\section{Introduction}

In many applications it is useful to have explicit error bounds
in the prime number theorem. \textsc{Rosser}~\cite{Rosser39,Rosser41} 
developed an analytic method which combines a numerical verification
of the \textsc{Riemann} hypothesis with a zero-free region and derived explicit estimates for some number theoretical functions. 
The aim of this paper is to find sharper bounds for the 
\textsc{Chebyshev}'s functions  $\psi(x)$, the logarithm of the least 
common multiple of all integers not exceeding $x$,  and 
$\theta(x)$, the product of all primes not exceeding $x$:
$$\theta(x)=\sum_{p\leqs x}\ln p,\hspace{1cm} 
\psi(x)=\sum_{\substack{{p,\alpha}\\{p^\alpha\leqs x}}}\ln p$$
where sum runs over primes $p$ and respectively  
over powers of primes $p^\alpha$.
The Prime Number Theorem could be written as follows: 
$$\psi(x)=x+o(x),\;\;x\rightarrow +\infty.$$
An equivalent formulation of the above theorem should be:
 for all $\eps>0$, there exists $x_0=x_0(\eps)$ 
such that 
$$|\psi(x)-x|<\eps\,x\hspace{0.5cm}\mbox{ for }x\geqs x_0$$
or 
$$|\theta(x)-x|<\eps\,x\hspace{0.5cm}\mbox{ for }x\geqs x_0.$$
Under Riemann Hypothesis (RH), \textsc{Schoenfeld} \cite{Schoenfeld:MathOfComp:Sharper} gives interesting results. 
Without the assumption of the RH, the results are not so accurate and depend on the knowledge about Riemann Zeta function.  
This article hangs up on some known results: 
the most important works on effective results have been shown by
 \textsc{Rosser} \&  \textsc{Schoenfeld} 
\cite{Rosser&Schoenfeld:Ill:Approx,Rosser&Schoenfeld:MathOfComp:Sharper,Schoenfeld:MathOfComp:Sharper}, 
 \textsc{Robin}~\cite{Robin:Acta:estim} \&  \textsc{Massias}~\cite{Massias:NP:effectiv}
 \and  \textsc{Costa Pereira}~\cite{Pereira:MathOfComp:Chebyshev}.

The proofs for estimates of $\psi(x)$ 
in \cite{Rosser&Schoenfeld:MathOfComp:Sharper} 
are based on the verification  of {\sc Riemann} hypothesis 
to a given height 
and an explicit zero-free region for $\zeta(s)$ 
whose form is essentially that the classical one of 
 \textsc{De la vall\'ee Poussin}.  \textsc{Rosser \& Schoenfeld}
have shown that the first  3~502~500 zeros of $\zeta(s)$ 
are on the critical strip. 
 \textsc{Van de Lune }{\it et al}~\cite{Lune86:ZerosZeta} have shown that 
the first 1~500~000~000 zeros  are on the critical strip.
Recently, \textsc{Wedeniwski}~\cite{zetagrid} then \textsc{Gourdon} \cite{Gourdon} manage to compute zeros in a parallel way and prove that the Riemann Hypothesis is true at least for first $10^{13}$ nontrivial zeros.

This will improve bounds~\cite{dusart:eps} for $\psi(x)$ and $\theta(x)$ for large values of $x$. 
We will prove the following results:
\begin{eqnarray*}
\theta(x)-x&<&\frac{1}{36\,260}x
        \hspace{1cm}\mbox{ for }x>0,\\
|\theta(x)-x|&\leqs&0.2\frac{x}{\ln^2 x}
        \hspace{1cm}\mbox{ for }x\geqs3\,594\,641.
\end{eqnarray*}
We apply these results on $p_k$, the $k^{th}$ prime, and $\theta(p_k)$. 
Let's denote by  $\ln_2 x$ for $\ln\ln x.$
The asymptotic expansion of $p_k$ is well known;
 \textsc{Cesaro}~\cite{Cesaro} then  \textsc{Cipolla}~\cite{Cipolla:determinazione} 
expressed it in 1902:
$$p_k=k\left\{\ln k+\ln_2 k-1+\frac{\ln_2 k-2}{\ln k}-\frac{\ln_2^2k-6\ln_2 k+11}{2\ln^2k}
+O\left(\left(\frac{\ln_2k}{\ln k}\right)^3\right)\right\}.$$
A more precise work about this can be find in \cite{Robin:MathOfComp:perm,Salvy:??:Fast}.
The results on $p_k$ are:
\begin{eqnarray*}
p_k&\leqs &k\left(\ln k+\ln_2 k-1+\frac{\ln_2 k-2}{\ln k}\right)
\hspace{1cm}\mbox{ for }k\geqs 688\,383,\\
p_k&\geqs &k\left(\ln k+\ln_2 k-1+\frac{\ln_2 k-2.1}{\ln k}\right)
\hspace{1cm}\mbox{ for }k\geqs 3.
\end{eqnarray*}

We use the above results to prove that, for $x\geqs 396\,738$, the interval 
$$\left[x,x+x/(25\ln^2 x)\right]$$ contains at least one prime.
Let's denote by $\pi(x)$ the number of primes not greater than $x$.
We show that
$$\frac{x}{\ln x}\left(1+\frac{1}{\ln x}\right)\;
{\stackrel{\leqs}{_{_{\scriptscriptstyle x\geqs599}}}}\;\pi(x)\;
{\stackrel{\leqs}{_{_{\scriptscriptstyle x>1}}}}\;
\frac{x}{\ln x}\left(1+\frac{1.2762}{\ln x}\right).$$ 
More precise results on $\pi(x)$ are also shown:
\begin{eqnarray*}
\pi(x)&\geqs&\frac{x}{\ln x-1} \mbox{ for } x\geqs5\,393,\\
\pi(x)&\leqs&\frac{x}{\ln x-1.1} \mbox{ for } x\geqs60\,184,\\
\pi(x)&\geqs&\frac{x}{\ln x}\left(1+\frac{1}{\ln x}+\frac{2}{\ln^2 x}\right)
         \mbox{ for }  x\geqs 88\,783,\\
\pi(x)&\leqs&\frac{x}{\ln x}\left(1+\frac{1}{\ln x}+\frac{2.334}{\ln^2 x}\right)
        \mbox{ for } x\geqs 2\,953\,652\,287.
\end{eqnarray*}

\section{Exact computation of $\theta$}

From the well-known identity
\begin{equation}
\label{DefPsi2}
\psi(x)=\sum_{k=1}^\infty \theta(x^{1/k}),
\end{equation}
we have
$$\theta(x)=\psi(x)-\sum_{k=2}^\infty \theta(x^{1/k}).$$
From some exact values of $\psi(x)$ computed by \cite{deleglise:psi}, we obtain Tables \ref{table_theta1} \& \ref{table_theta2} (Exact values of $\theta(x)$)

\section{On the difference between $\psi$ and $\theta$}

As $\theta(2^-)=0$, the summation (\ref{DefPsi2}) ends:
$$\psi(x)=\sum_{k=1}^{\lfloor\frac{\ln x}{\ln 2}\rfloor}\theta(x^{1/k})=\theta(x)+\theta(\sqrt{x})+\sum_{k=3}^{\lfloor\frac{\ln x}{\ln 2}\rfloor}\theta(x^{1/k}).$$

\subsection{Lower Bound}
\begin{prop}
For $x\geqs121$, we have
\begin{equation}
\label{MinDiffPsi}
0.9999\sqrt{x}<\psi(x)-\theta(x)
\end{equation}
\end{prop}

\begin{proof}
By Theorem 24 of \cite{Rosser&Schoenfeld:Ill:Approx} p.73, (\ref{MinDiffPsi}) is verified for $121\leqs x \leqs 10^{16}$.
Now by \cite{Pereira:MathOfComp:Chebyshev} p.~211,
$$\psi(x)-\theta(x)=\psi(\sqrt{x})+\sum_{k\geqs 1}\theta(x^\frac{1}{2k+1}),$$
hence
$$\psi(x)-\theta(x)\geqs\psi(\sqrt{x})+\theta(x^{1/3}).$$
By Theorem~19 of \cite{Rosser&Schoenfeld:Ill:Approx} p.72, we have
$$\theta(x^{1/3})>\sqrt[3]{x}-2x^{1/6}\mbox{ for }(1423)^3\leqs x\leqs (10^8)^3,$$
and we have for $x\geqs \exp(2b)$,
$$\psi(\sqrt{x})>\sqrt{x}-\eps_b \sqrt{x}=0.9999\sqrt{x}+(0.0001-\eps_b) \sqrt{x}.$$
where $\eps_b$ can be find in Table~\ref{table_eps} (or Table p.358 of \cite{Schoenfeld:MathOfComp:Sharper}).
We verify that
$$(0.0001-\eps_b)\sqrt{x}+\sqrt[3]{x}-2x^{1/6}>0$$
for $10^{16}\leqs x\leqs e^{50}$ by intervals (we use $b=18.42,\;20,\;22$).
For $y\geqs e^{25}$, Table~\ref{table_eps} gives $|\psi(y)-y|<0.00007789y.$ Hence we have by Th.13 of \cite{Rosser&Schoenfeld:Ill:Approx},
\begin{eqnarray*}
|\theta(y)-y|&\leqs&|\psi(y)-y|+|\theta(y)-\psi(y)|<0.00007789y+1.43\sqrt{y}\\
			&<&0.00009y 
\end{eqnarray*}
For $x>e^{50}$, we apply the previous result with $y=\sqrt{x}$ to obtain
$$\psi(x)-\theta(x)>\theta(\sqrt{x})\geqs 0.9999\sqrt{x}.$$
\end{proof}

\subsection{Upper Bound}
\begin{prop}
\label{maj_diff_psi_theta}
For $x>0$,
$$\psi(x)-\theta(x)<1.00007\sqrt{x}+1.78\sqrt[3]{x}.$$
\end{prop}

\begin{proof}
We use (\ref{DiffPsiSqrt}) and Proposition~\ref{prop_eta0}.

\end{proof}
\begin{lemma}
For $x>0$, we have
\begin{equation}
\label{DiffPsiSqrt}
\psi(x)-\theta(x)-\theta(\sqrt{x})<1.777745 x^{1/3}
\end{equation}
\end{lemma}

\begin{proof}
For $x>0$, we have $\theta(x)<1.000081x$ by \cite{Schoenfeld:MathOfComp:Sharper} p.360. Hence
\begin{eqnarray*}
\sum_{k=3}^{\lfloor\frac{\ln x}{\ln 2}\rfloor}\theta(x^{1/k})&<&1.000081\sum_{k=3}^{\lfloor\frac{\ln x}{\ln 2}\rfloor}x^{1/k}\\
&<&1.000081\left(x^{1/3}+\left(\left\lfloor\frac{\ln x}{\ln 2}\right\rfloor-4\right)x^{1/4}\right)\\
&<&1.2\;x^{1/3}\mbox{ for }x>(10^{11})^3.
\end{eqnarray*}
For small values, we have (\ref{DiffPsiSqrt}) by direct computation (Maximal value reaches for $x$=2401).
\end{proof}

\section{Useful Bounds}
\begin{eqnarray}
\label{bound1} p_k\leqs k\ln p_k &\mbox{ for }&k\geqs 4,\\
\label{bound2} \ln p_k\leqs \ln k+\ln_2 k +1&\mbox{ for }&k\geqs 2.
\end{eqnarray}

\begin{proof}
We deduce (\ref{bound1}) from $\pi(x)>\frac{x}{\ln x}$ (Corollary~1 of \cite{Rosser&Schoenfeld:Ill:Approx}).
By Theorem~3 of \cite{Rosser&Schoenfeld:Ill:Approx}, we have
$p_k<k(\ln k+\ln_2 k-1/2)$ hence $p_k<e k\ln k$ for $k\geqs 2$.
\end{proof}
\section{On the differences between $\theta$ and identity function}

\begin{prop}
\label{prop_eta0}
$\theta(x)-x<\frac{1}{36\,260} x$ for $x>0$.
\end{prop}
\begin{proof}
By table~\ref{table_012}, we have
$\theta(x)<x$ up to $8\cdot10^{11}$.
With (\ref{MinDiffPsi}) and $8\cdot10^{11}\leqs x\leqs e^{28}$,
$$\theta(x)<\psi(x)-0.9999\sqrt{x}<(1.00002841-0.9999/\sqrt{e^{28}})x<1.00002758x.$$
We conclude by computing $\eps_{28}\leqs 0.00002224$.
\end{proof}

\begin{theorem}
\label{th:eta_k}
We have
$$|\theta(x)-x|<\eta_k \frac{x}{\ln^k x}\quad\mbox{ for }x\geqs x_k$$
with 
$$\begin{array}{|r|c|c|c|c|c|c|c|}
\hline
k      & 0 & 1      &  1              &   2   & 2 & 2 & 2 \\
\hline
\eta_k & 1 & 1.2323 & 0.001           & 3.965 & 0.2 &0.05  &0.01  \\
\hline
x_k    & 1 & 2      & 908\,994\,923 &  2    &3\,594\,641 & 122\,568\,683&7\,713\,133\,853\\
\hline
\end{array}$$
and 
$$\begin{array}{|r|c|c|c|c|c|}
\hline
k&    3      &3  &3 &3   &4\\
\hline
\eta_k&20.83 &10 &1 &0.78  &1300\\
\hline
x_k&2    &32\,321 &89\,967\,803 & 158\,822\,621 &2\\
\hline
\end{array}$$
\end{theorem}

\begin{proof}
We use the estimates of $|\psi(x)-x|$ with Proposition~\ref{maj_diff_psi_theta}.
In particular, we can choose $\eta_2=0.05$ because
$$\scriptstyle(0.00006788+1.00007/\sqrt{10^{11}}+1.78/(10^{11})^{2/3})*26^2
< 0.04809.$$ We obtain Tables~\ref{table_eta1} \& \ref{table_eta2} step by step up to $b=5000$.
For each line, the value is valid between $b_i$ and $b_{i+1}$. Hence, by example, $\eta_2=4.42E-3$ should be chosed for $x\geqs e^{32}$.

Using Theorem~1.1 of \cite{dusart:eps},
we have $\eta_k\geqs\sqrt{8/\pi}(\sqrt{\ln(x_0)/R})^{1/2}\cdot e^{-\sqrt{\ln(x_0)/R}}\cdot\ln^k(x_0)$ to
obtain for $x\geqs x_0= \exp(5000)$,
\begin{eqnarray*}
\eta_0&=&1.196749447941324988148958471 E-12,\\
\eta_1&=&0.000000005983747239706624940744792353,\\
\eta_2&=& 0.00002991873619853312470372396176,\\
\eta_3&=& 0.1495936809926656235186198088,\\
\eta_4&=&747.9684049633281175930990441.
\end{eqnarray*}

Specials constants:\\
for $x\geqs 1$, $\eta_0<(1-\theta(1^-))/1=(2-\theta(2^-))/2=1$.\\
for $x\geqs 2$, $\eta_1<(11-\theta(11^-))/11\cdot\ln(11)\approx 1.23227674$.\\
for $x\geqs 2$, $\eta_2<(59-\theta(59^-))/59\cdot\ln^2(59)\approx 3.964809$\\
for $x\geqs 2$, $\eta_3<(1423-\theta(1423^-))/1423\cdot\ln^3(1423)\approx 20.8281933$
\end{proof}
\section{Some applications on number theory functions}

\subsection{Estimates of primes}
\subsubsection{Estimates of $\theta(p_k)$}

We have an asymptotic development of $\theta(p_k)$:
$$\theta(p_k)=\li^{-1}(k)+O(k^{1/2}\ln^{3/2}k)$$ whose the first terms by \cite{Cipolla:determinazione} are
$$\theta(p_k)=k\left(\ln k+\ln_2 k-1 +\frac{\ln_2 k-2}{\ln k}-\frac{\ln_2^2k-6\ln_2k+11}{2\ln^2 k}+0\left(\frac{\ln_2^3 k}{\ln^3 k}\right)\right)$$

\begin{remark}
We have
\begin{equation}
\label{maj_theta_pk}
\theta(p_k)\leqs  k\left(\ln k+\ln_2 k-1+\frac{\ln_2 k-2}{\ln k}\right)\quad\mbox{ for } k\geqs 198.
\end{equation}
 by Th.~B(v) of \cite{Massias:NP:effectiv}.
\end{remark}

\begin{prop}
\label{min_theta_pk}
$$\theta(p_k)\geqs  k\left(\ln k+\ln_2 k-1+\frac{\ln_2 k-2.050735}{\ln k}\right)\quad\mbox{ for }p_k\geqs 10^{11}$$
$$\theta(p_k)\geqs  k\left(\ln k+\ln_2 k-1+\frac{\ln_2 k-2.04}{\ln k}\right)\quad\mbox{ for }p_k\geqs 10^{15}$$
\end{prop}

\begin{proof}
Let $f_\beta$ defined by
$$n\mapsto n\left(\ln n+\ln_2 n-1+\frac{\ln_2 n-\beta}{\ln n}\right).$$
We want to prove that $\theta(p_n)\geqs f_\beta(n)$.
Define $h_a$ by $h_a(n):=n\left(\ln n+\ln_2 n -a\right).$
Suppose there exist $a$ such that $p_k\geqs h_a(k)$ for $k\geqs k_0$.
Hence 
$$\theta(p_k)-\theta(p_{k_0})=\sum_{n=k_0+1}^k\ln p_n\geqs \sum_{n=k_0+1}^k\ln h_a(n).$$
We have $f'_\beta\leqs \ln h_a$ if
\begin{equation}
\label{eq:beta2}
\frac{\ln_2 n-\beta+1}{\ln n}-\frac{\ln_2 n-\beta-1}{\ln^2 n}\leqs \ln\left(1+\frac{\ln_2 n-a}{\ln n}\right).
\end{equation}

We can rewrite (\ref{eq:beta2}) as
\begin{equation}
\label{eq:beta3}
\beta(1-1/\ln k)\geqs 1+\ln_2 k-\ln\left(1+\frac{\ln_2 k -a}{\ln k}\right)\ln k -\frac{\ln_2 k -1}{\ln k}.
\end{equation}
For $a\in[0.95,1]$ and  $t\geqs 22$, the function 
$t\mapsto (\ln t-t\ln\left(1+\frac{\ln t -a}{t}\right) -\frac{\ln t -1}{t})/(1-1/t)$ is decreasing.
 
By \cite{dusart:pk}, we can choose $a=a_0=1$. 
For $k\geqs e^{100}$, the value $\beta=2.048$ satisfies (\ref{eq:beta3}).

For $\pi(10^{11})\leqs k \leqs e^{100}$, the value $\beta_0=2.094$ satisfies (\ref{eq:beta3}). 
Hence
$$\theta(p_k)\geqs k\left(\ln k+\ln_2 k-1+\frac{\ln_2 k-\beta_0}{\ln k}\right).$$
Then
$p_k\geqs \theta(p_k)-\eta_2\frac{k}{\ln k}$ by (\ref{th:eta_k}) \& (\ref{bound1}), hence $p_k\geqs h_{a_1}(k)$ with $a_1=1-\frac{\ln_2 k-(\beta_0+\eta_2)}{\ln k}$.
Splitting the interval of $k$, we use different values of $a$ 
with adapted values of $\eta_2$. By itering the process, we obtain $\beta=2.050735$ for $k\geqs k_0=\pi(10^{11})$. This value of $\beta$ verifies $\theta(p_{k_0})\geqs f_{\beta}(k_0)$.

By same way, we obtain $\beta=2.038$ for $k\geqs 10^{15}$.
\end{proof}


\begin{prop}
\label{prop:upper:thetaPk}
For $k\geqs 781$,
$$\theta(p_k)\leqs k\left(\ln k+\ln_2 k-1 +\frac{\ln_2 k-2}{\ln k}-\frac{0.782}{\ln^2 k}\right)$$
\end{prop}

\begin{proof}
Use Lemma~\ref{lem:upper:pk} and Lemma~\ref{lem:upper:thetaPk}.
\end{proof}

\begin{lemma}
\label{lem:upper:thetaPk}
Let two integers $k_0,k$ and $\gamma>0$ real.
Suppose that for $k_0\leqs n \leqs k$,
$$p_n\leqs n\left(\ln n+\ln_2 n-1+\frac{\ln_2 n-1.95}{\ln n}\right).$$
Let $s(k)=k\left(\ln k+\ln_2 k-1 +\frac{\ln_2 k-2}{\ln k}-\frac{\gamma}{\ln^2 k}\right)$.
Let $f(k)=s(k)-(\ln k +\ln_2 k+1).$ If $\theta(p_{k_0-1})\leqs f(k_0)$ then $\theta(p_k)\leqs s(k)$ for all $k\geqs k_0$.
\end{lemma}

\begin{proof}
Let $S_a(n)$ be an upper bound for $p_n$ for $k_0\leqs n\leqs k$ where
$$S_a(n)=n\left(\ln n+\ln_2 n-1+\frac{\ln_2 n-a}{\ln n}\right).$$
Now, for $2\leqs k_0\leqs k$, we write
$$\theta(p_{k-1})-\theta(p_{k_0}-1)=\sum_{n=k_0}^{k-1}\ln p_n\leqs\sum_{n=k_0}^{k-1}\ln S_a(n)\leqs\int_{k_0}^k\ln S_a(n)dn.$$

We need to prove that $\ln S_a(n)\leqs f'(n)$.

We have $$\ln S_a(n)= \ln n+\ln_2 n+\ln(1+u(n))$$ with $u(n)=\frac{\ln_2 n-1}{\ln n}+\frac{\ln_2 n-a}{\ln^2 n}$ and
$$f'(n)=\ln n+\ln_2 n+\frac{\ln_2 n-1}{\ln n}-\frac{\ln_2 n+\gamma-3}{\ln^2 n}+\frac{2\gamma}{\ln^3 n}-\frac{1}{n}(1+1/\ln n).$$

Let $\beta<1/2$ such that $\ln(1+u(n))\leqs u(n)-\beta u^2(n)$ for $n\geqs k_0$.
Then $\ln S_a(n)\leqs f'(n)$ if
$$\beta\left(\frac{\ln_2 n-1}{\ln n}+\frac{\ln_2 n-a}{\ln^2 n}\right)^2 -\frac{2\ln_2 n+\gamma-3-a}{\ln^2 n}+2\gamma/\ln^3 n-1/n-1/(n\ln n)\geqs 0,$$ that we can simplify in
$$\frac{A}{\ln^2 n}+2\frac{B}{\ln^3 n}+\beta\frac{\ln_2^2 n -2a\ln_2 n+a^2}{\ln^4 n}-1/n-1/(n\ln n)\geqs 0$$
where 
$$A=\beta\ln_2^2 n-2(\beta+1)\ln_2 n+3+a+\beta-\gamma$$
$$B=\beta\ln_2^2 n-\beta(a+1)\ln_2 n+a\beta+\gamma$$
We have $1/n+1/(n\ln n)\leqs 0.02/\ln^3 n$ for $n\geqs 10^5$.

We study each parts, denoting $\ln_2 n$ by $X$:
\begin{itemize}
\item[$\bullet$] $\beta X^2-2(\beta+1) X +3+a+\beta-\gamma\geqs 0$ for all $X$ if $\gamma-a-1+1/\beta\leqs 0$,
\item[$\bullet$] $X^2-(a+1)X+ (a+\gamma/\beta+0.02)\geqs 0$ for all $X$ if $a^2-2a+1-4(\gamma/\beta+0.02)\leqs 0$,
\item[$\bullet$] $X^2-2aX+a^2=(X-a)^2\geqs 0$.
\end{itemize}

We choose $\gamma$ such that $\gamma-a-1+1/\beta=0$.
We choose $\beta=\frac{u(k_0)-\ln(1+u(k_0))}{u^2(k_0)}$.
With $a=1.95$ and $k_0=178974$, we have $\beta=0.461291475\cdots$ and $\gamma=0.78217325\cdots$.

Hence
$\theta(p_{k-1})-f(k)\leqs \theta(p_{k_0}-1)-f(k_0)$.
As $\theta(p_{k_0}-1)\leqs f(k_0)$, we have $\theta(p_{k-1})-f(k)\leqs 0$.
We obtain the upper bound
$\theta(p_k)=\theta(p_{k-1})+\ln p_k \leqs f(k)+\ln p_k<s(k)$ by (\ref{bound2}).

\end{proof}

\subsubsection{Estimates of $p_k$}

\begin{lemma}
\label{lem:upper:pk}
For $k\geqs 178\,974$,
$$p_k\leqs k\left(\ln k+\ln_2 k-1+\frac{\ln_2 k-1.95}{\ln k}\right).$$
\end{lemma}
\begin{proof}
Substituing $x$ by $p_k$ in  $|\theta(x)-x|\leqs \eta_2\frac{x}{\ln^2 x}$,
we obtain
$$|p_k- \theta(p_k)|\leqs\eta_2\frac{p_k}{\ln^2 p_k}.$$
By (\ref{bound1}), we have
$\frac{p_k}{\ln^2 p_k}\leqs \frac{k}{\ln k}$ and 
\begin{equation}
\label{diff_theta_pk_pk}
|p_k - \theta(p_k)|\leqs \eta_2\frac{k}{\ln k}.
\end{equation}
Using the upper bound (\ref{maj_theta_pk}) of $\theta(p_k)$,
we have
$$p_k\leqs k\left(\ln k+\ln_2 k-1+\frac{\ln_2 k-2+\eta_2}{\ln k}\right).$$
We use $\eta_2=0.05$ for $x\geqs 10^{11}$.
\end{proof}

\begin{prop}
\label{prop:upper:pk}
For $k\geqs 688\,383$,
$$p_k\leqs k\left(\ln k+\ln_2 k-1+\frac{\ln_2 k-2}{\ln k}\right).$$
\end{prop}

\begin{proof}
Use Proposition~\ref{prop:upper:thetaPk} with  $\eta_3=0.78$ of Theorem~\ref{th:eta_k} for $\ln p_k>27$. A computer verification concludes the proof.
\end{proof}

\begin{prop}
For $k\geqs 3$,
$$p_k\geqs k\left(\ln k+\ln_2 k-1+\frac{\ln_2 k-2.1}{\ln k}\right).$$
\end{prop}
\begin{proof}
Using (\ref{diff_theta_pk_pk}),
we have
$$p_k\geqs \theta(p_k)-\eta_2\frac{k}{\ln k}.$$
By Proposition~\ref{min_theta_pk} and $\eta_2=0.04913$, we conclude the proof.
\end{proof}


\subsubsection{Smallest Interval containing primes}
We already know the result of
{\sc Schoenfeld} \cite{Schoenfeld:MathOfComp:Sharper} showing that, 
for $x\geqs2010759.9$,
the interval $]x,x+x/16597[$ contains at least one prime.
We improve this result with the following proposition.
You can see also \cite{Ramare:JNT:short}.

\begin{prop}
For all $x\geqs 396\,738$, there exists a prime  $p$ such that
$$x<p\leqs x\left(1+\frac{1}{25\ln^2 x}\right).$$
\end{prop}

This result is better that {\sc Rosser} \& {\sc Schoenfeld}'s one
for $x\geqs e^{25.77}$. The method used in \cite{Ramare:JNT:short} gives better results (if we compare with the same order of $k$, i.e. $k=0$).

\begin{proof}
Let $0<f(x)<1$ for $x\geqs x_0$.
\begin{eqnarray*}
\theta\left(\frac{1}{1-f(x)}x\right)-\theta(x)&\geqs&\frac{1}{1-f(x)}x
-\eta_k\frac{\frac{x}{1-f(x)}}{\ln^k\left(\frac{x}{1-f(x)}\right)}
-\left(x+\eta_k\frac{x}{\ln^k x}\right)\\
&>&\left(\frac{1}{1-f(x)}-1\right)x-2\eta_k\left(\frac{1}{1-f(x)}\right)\frac{x}{\ln^k x}
\end{eqnarray*}
Choose $f(x)=\frac{2\eta_k}{\ln^k x}$ hence
$$\theta\left(\frac{1}{1-\frac{2\eta_k}{\ln^k x}}x\right)-\theta(x)>0.$$
For $k=2$, we have $n_2=0.0195$  and
$\frac{1}{1-2\cdot0.0195/\ln^2 x}\leqs 1+1/(25\ln^2 x)$ for $\ln x\geqs 28$.
According to \cite{Schoenfeld:MathOfComp:Sharper} p.~355,
$$p_{n+1}-p_n\leqs 652\mbox{ for }p_n\leqs2.686\cdot10^{12},$$
hence the result is also valid from $x\geqs 3.8\cdot 10^6$.
\end{proof}

\subsection{Estimates of function $\pi$}
Remember that
$$\pi(x)=\frac{x}{\ln x}\left(1+\frac{1}{\ln x}+\frac{2}{\ln^2 x}
+O\left(\frac{1}{\ln^3 x}\right)\right).$$

\begin{theorem}
\label{Ross:th:pi}
\begin{eqnarray}
\label{pi_ordre1}
\frac{x}{\ln x}\left(1+\frac{1}{\ln x}\right)\leqs&\pi(x)&\leqs\frac{x}{\ln x}\left(1+\frac{1.2762}{\ln x}\right)\\
\nonumber \scriptstyle{ x\geqs 599}&&{\scriptstyle x>1}\\
\nonumber\lefteqn{\mbox{(the value 1.2762 is chosen for }x=p_{258}=1627). }\\ %
\nonumber&&\\
\label{pi_ordre1_5}
\frac{x}{\ln x-1}\leqs&\pi(x)&\leqs\frac{x}{\ln x-1.1}\\
\nonumber \scriptstyle x\geqs5\,393&&\scriptstyle x\geqs60\,184\\
\nonumber&&\\
\label{pi_ordre2}\frac{x}{\ln x}\left(1+\frac{1}{\ln x}+\frac{2}{\ln^2 x}\right)\leqs&
\pi(x)&\leqs\frac{x}{\ln x}\left(1+\frac{1}{\ln x}+\frac{2.334}{\ln^2 x}\right)\\
\nonumber \scriptstyle x\geqs 88\,783&&\scriptstyle x\geqs2\,953\,652\,287
\end{eqnarray}

\end{theorem}

\begin{proof}
We consider the last inequality.
Let
$$x_{0}=10^{11},\quad K=\pi(x_{0})-\frac{\theta(x_{0})}{\ln x_{0}}.$$
Write
$$J(x;\eta_k)=K+\frac{x}{\ln x}+\eta_k\frac{x}{\ln ^{k+1} x}
+\int_{x_{0}}^x \left(\frac{1}{\ln^2 y}+\frac{\eta_k}{\ln^{k+2} y}\right)dy$$
Since
$$\pi(x)=\pi(x_{0})-\frac{\theta(x_{0})}{\ln x_{0}}
+\frac{\theta(x)}{\ln x}+\int_{x_{0}}^{x}\frac{\theta(y)dy}{y\ln^2 y}$$
and $|\theta(x)-x|\leqs \eta_k\frac{x}{\ln^k x}$ for $x\geqs x_{0}$,
we have, for $x\geqs x_{0}$,
$$J(x;-\eta_k)\leqs \pi(x)\leqs J(x;\eta_k).$$
Write $M(x;c)=\frac{x}{\ln x}\left(1+\frac{1}{\ln x}+\frac{c}{\ln^2 x}
\right)$ for upper bound's function for $\pi(x)$.
Let's write the derivatives of $J(x;a)$ and of $M(x;c)$ with 
respect to $x$:
$$J'(x;a)=\frac{1}{\ln x}+\frac{\eta_k}{\ln^{k+1} x}-k\frac{\eta_k}{\ln^{k+2} x},$$
$$M'(x;c)=\frac{1}{\ln x}+\frac{c-2}{\ln^3 x}-\frac{3c}{\ln^4 x}.$$
For $k=2$, we must choose $c\geqs (2+\eta_2-2\eta_2/\ln x_0)/(1-3/\ln x_0)$ to have $J'<M'$ for $x\geqs x_0$.
With $\eta_2=0.05$, we choose $c=2.321$. We verify by computer that
$J(10^{11};0.05)<M(10^{11};2.334)$. 

By direct computation for small values of $x$ to obtain
$$\pi(x)<\frac{x}{\ln x}\left(1+\frac{1}{\ln x}+\frac{2.334}{\ln^2 x}
\right)\quad\mbox{ for }x\geqs 2\,953\,652\,287.$$
Now write
$$m(x;d)=\frac{x}{\ln x}\left(1+\frac{1}{\ln x}+\frac{d}{\ln^2 x}
\right).$$
We study the derivatives: we choose $k=3$, $d=2$ and $\eta_3(1-3/\ln x)<6$ to have $J'>m'$. 
 As $m(x_{0};2)<J(x_{0};-6)$ and by direct 
computation for small values, we obtain
$$ \pi(x)>\frac{x}{\ln x}\left(1+\frac{1}{\ln x}+\frac{2}{\ln^2 x}
\right)\quad\mbox{ for }x\geqs 88\,783.$$

The others inequalities follows: 
$(\ref{pi_ordre2})\Rightarrow(\ref{pi_ordre1_5})\Rightarrow(\ref{pi_ordre1})$ for large $x$.
\end{proof}

\subsection{Estimates of sums over primes}
Let $\gamma$ be Euler's constant ($\gamma\approx0.5772157$).

\begin{theorem}
\label{Rosser:th:InvP}
Let $B= \gamma+
\sum_{p}\left(\ln(1-1/p)+1/p\right)\approx0.26149~72128~47643$.
For $x>1$,
$$-\left(\frac{1}{10\ln^2 x}+\frac{4}{15\ln^3x}\right)\leqs\sum_{p\leqs x}\frac{1}{p}-\ln_{2}x -B
.$$
For $x\geqs10372$,
$$\sum_{p\leqs x}\frac{1}{p}-\ln_{2}x -B\leqs\frac{1}{10\ln^2 
x}+\frac{4}{15\ln^3x}.$$
\end{theorem}

\begin{proof}
By (4.20) of \cite{Rosser&Schoenfeld:Ill:Approx},
$$\sum_{p\leqs x}\frac{1}{p}=\ln_{2}x +B+\frac{\theta(x)-x}{x\ln x}
-\int_{x}^\infty \frac{(\theta(y)-y)(1+\ln y)}{y^2\ln^2 y}dy.$$
Hence
$$|\sum_{p\leqs x}\frac{1}{p}-\ln_{2}x -B|\leqs
\frac{|\theta(x)-x|}{x\ln x}
+\int_{x}^\infty \frac{|\theta(y)-y|(1+\ln y)}{y^2\ln^2 y}dy.$$
As $|\theta(x)-x|\leqs \eta_k x/\ln^k x$ (Theorem~\ref{th:eta_k}) and
$$\int_{x}^\infty\frac{1+\ln y}{y\ln^{k+2} y}dy=\frac{1}{k\ln^k x}
+\frac{1}{(k+1)\ln^{k+1} x},$$
we have the result
\begin{equation}
\label{eq:sumInvP}
\left|\sum_{p\leqs x}\frac{1}{p}-\ln_{2}x -B\right|
\leqs \frac{\eta_k/k}{\ln^k x}+\frac{\eta_k(1+\frac{1}{k+1})}{\ln^{k+1} x}.
\end{equation}
For $k=2$ and $\eta_2=0.2$, the result is valid for $x\geqs 3594641$. We conclude by computer's check.
\end{proof}

\begin{theorem}
\label{Rosser:th:lnPInvP}
Let $E= -\gamma-\sum_{n=2}^{\infty}\sum_{p}(\ln p)/p^n
\approx-1.33258~22757~33221$.
For $x>1$,
$$-\left(\frac{0.2}{\ln x}+\frac{0.2}{\ln^2x}\right)
\leqs\sum_{p\leqs x}\frac{\ln p}{p}-\ln x -E.$$
For $x\geqs 2974$,
$$\sum_{p\leqs x}\frac{\ln p}{p}-\ln x -E\leqs\frac{0.2}{\ln x}
+\frac{0.2}{\ln^2x}.$$
\end{theorem}

\begin{proof}
By (4.21) of \cite{Rosser&Schoenfeld:Ill:Approx},
$$\sum_{p\leqs x}\frac{\ln p}{p}=\ln x +E+\frac{\theta(x)-x}{x}
-\int_{x}^\infty \frac{\theta(y)-y}{y^2}dy.$$
Hence
$$|\sum_{p\leqs x}\frac{\ln p}{p}-\ln x -E|\leqs\frac{|\theta(x)-x|}{x}
+\int_{x}^\infty \frac{|\theta(y)-y|}{y^2}dy.$$
As
$$\int_{x}^\infty\frac{dy}{y\ln^k y}=\frac{1}{(k-1)\ln^{k-1} x},$$
Theorem~\ref{th:eta_k} yields the result for $x\geqs 3594641$ with $k=2$. 
We conclude by computer's check.
\end{proof}

\subsection{Estimates of products over primes} 
\begin{theorem}
\label{Rosser:th:ProdP}
For $x>1$,
$$\prod_{p\leqs x}\left(1-\frac{1}{p}\right)<\frac{e^{-\gamma}}{\ln x}
\left(1+\frac{0.2}{\ln^2 x}\right)$$
and for $x\geqs 2\,973$,
$$\frac{e^{-\gamma}}{\ln x}
\left(1-\frac{0.2}{\ln^2 x}\right)
<\prod_{p\leqs x}\left(1-\frac{1}{p}\right)
$$

For $x>1$,
$${e^{\gamma}}{\ln x}
\left(1-\frac{0.2}{\ln^2 x}\right)<\prod_{p\leqs x}\frac{p}{p-1}.
$$
and for $x\geqs 2\,973$,
$$\prod_{p\leqs x}\frac{p}{p-1}<{e^{\gamma}}{\ln x}
\left(1+\frac{0.2}{\ln^2 x}\right).$$
\end{theorem}

\begin{proof}
By definition of $B$ and (\ref{eq:sumInvP}), we have
$$\left|-\gamma-\ln_{2}x-\sum_{p>x}\frac{1}{p}-\sum_{p}\ln(1-1/p)\right|
\leqs\frac{\eta_k/k}{\ln^k x}+\frac{\eta_k(1+\frac{1}{k+1})}{\ln^{k+1} x}.$$
Let 
$S=\sum_{p>x}\left(\ln(1-1/p)+1/p\right)=
-\sum_{n=2}^\infty\frac{1}{n}\sum_{p>x}\frac{1}{p^n}.$
We have
$$-\gamma-\ln_{2}x-\sum_{p\leqs x}\ln(1-1/p)-S
\geqs-\frac{\eta_k}{k\ln^k x}-\frac{(k+2)\eta_k}{(k+1)\ln^{k+1} x}.$$
Take the exponential of both sides to obtain
$$\prod_{p\leqs x}\left(1-\frac{1}{p}\right)\leqs 
\frac{e^{-\gamma}}{\ln x}\exp\left( -S+\frac{\eta_k}{k\ln^k x}
+\frac{(k+2)\eta_k}{(k+1)\ln^{k+1} x}\right).$$
We use lower bound for S given in \cite{Rosser&Schoenfeld:Ill:Approx} 
p.~87:
$$-S<\frac{1.02}{(x-1)\ln x}.$$
Hence, for $k=2$, $\eta_2=0.2$ and $x\geqs 3\,594\,641$,
$$\prod_{p\leqs x}\left(1-\frac{1}{p}\right)\leqs 
\frac{e^{-\gamma}}{\ln x}\exp(0.11/\ln^2 x).$$
We have also
$$\prod_{p\leqs x}\frac{p-1}{p}\geqs 
{e^{\gamma}}{\ln x}\exp(-0.11/\ln^2 x).$$

In the same way, as 
$$-\gamma-\ln_{2}x-\sum_{p\leqs x}\ln(1-1/p)-S
\leqs\frac{\eta_k}{k\ln^k x}+\frac{(k+2)\eta_k}{(k+1)\ln^{k+1} x},$$
we obtain the others inequalities since $S\leqs0$.
\end{proof}

\begin{table}[!hp]
\caption{Values of $\theta(x)$ for $10^{6} \leq x \leq 10^{10}$}
\protect{\label{table_theta1}}
\begin{center}
\normalsize
\begin{tabular}{||c|r|r||} \hline
   x    &               $\theta(x)$               & $\psi(x)-\theta(x)$ \\
\hline
$ 1E+06 $ & $998484.175026 $ & $ 1102.422470 $\\
$ 2E+06 $ & $1998587.722137 $ & $ 1527.324070 $\\
$ 3E+06 $ & $2998107.530452 $ & $ 1892.449541 $\\
$ 4E+06 $ & $3997323.492084 $ & $ 2167.364713 $\\
$ 5E+06 $ & $4998571.086801 $ & $ 2400.053428 $\\
$ 6E+06 $ & $5996983.791998 $ & $ 2665.785692 $\\
$ 7E+06 $ & $6997751.998535 $ & $ 2823.187880 $\\
$ 8E+06 $ & $7997057.246292 $ & $ 3064.486910 $\\
$ 9E+06 $ & $8997625.570065 $ & $ 3224.678815 $\\
$ 1E+07 $ & $9995179.317856 $ & $ 3360.085490 $\\
$ 2E+07 $ & $19995840.882153 $ & $ 4759.143006 $\\
$ 3E+07 $ & $29994907.240152 $ & $ 5797.041942 $\\
$ 4E+07 $ & $39994781.014188 $ & $ 6699.200805 $\\
$ 5E+07 $ & $49993717.861720 $ & $ 7489.482783 $\\
$ 6E+07 $ & $59991136.134174 $ & $ 8172.843038 $\\
$ 7E+07 $ & $69991996.348980 $ & $ 8786.853393 $\\
$ 8E+07 $ & $79988578.197461 $ & $ 9388.261229 $\\
$ 9E+07 $ & $89985867.940581 $ & $ 9992.336337 $\\
$ 1E+08 $ & $99987730.018022 $ & $ 10512.778605 $\\
$ 2E+08 $ & $199982302.435783 $ & $ 14725.068769 $\\
$ 3E+08 $ & $299981378.219200 $ & $ 18000.443659 $\\
$ 4E+08 $ & $399982033.338736 $ & $ 20744.718991 $\\
$ 5E+08 $ & $499983789.813730 $ & $ 23200.125087 $\\
$ 6E+08 $ & $599976282.577668 $ & $ 25426.013243 $\\
$ 7E+08 $ & $699976911.639135 $ & $ 27402.910397 $\\
$ 8E+08 $ & $799969331.209833 $ & $ 29215.380561 $\\
$ 9E+08 $ & $899953849.181850 $ & $ 30963.754721 $\\
$ 1E+09 $ & $999968978.577566 $ & $ 32617.412861 $\\
$ 2E+09 $ & $1999941083.684486 $ & $ 46075.813369 $\\
$ 3E+09 $ & $2999937036.966284 $ & $ 56255.144708 $\\
$ 4E+09 $ & $3999946136.165586 $ & $ 64858.831531 $\\
$ 5E+09 $ & $4999906575.362844 $ & $ 72411.275590 $\\
$ 6E+09 $ & $5999930311.133705 $ & $ 79301.775139 $\\
$ 7E+09 $ & $6999917442.519773 $ & $ 85715.065356 $\\
$ 8E+09 $ & $7999890792.693956 $ & $ 91420.172461 $\\
$ 9E+09 $ & $8999894889.497541 $ & $ 97066.566501 $\\
$ 1E+10 $ & $9999939830.657757 $ & $ 102289.175716 $\\
\hline
\end{tabular}
\end{center}
\end{table}

\pagebreak 
\begin{table}[!hp]
\caption{Values of $\theta(x)$ for $10^{10} \leq x \leq 10^{15}$}
\protect{\label{table_theta2}}
\begin{center}
\normalsize
\begin{tabular}{||c|r|r||} \hline
   x    &      $\theta(x)$     & $\psi(x)-\theta(x)$ \\
\hline
$ 1E+10 $ & $9999939830.657757 $ & $ 102289.175716 $\\
$ 2E+10 $ & $19999821762.768212 $ & $ 144339.622582 $\\
$ 3E+10 $ & $29999772119.815419 $ & $ 176300.955450 $\\
$ 4E+10 $ & $39999808348.775748 $ & $ 203538.541084 $\\
$ 5E+10 $ & $49999728380.731899 $ & $ 227474.729168 $\\
$ 6E+10 $ & $59999772577.550769 $ & $ 249003.320704 $\\
$ 7E+10 $ & $69999769944.203933 $ & $ 268660.720820 $\\
$ 8E+10 $ & $79999718357.195652 $ & $ 287365.266118 $\\
$ 9E+10 $ & $89999644656.090911 $ & $ 304250.688854 $\\
$ 1E+11 $ & $99999737653.107445 $ & $ 320803.322857 $\\
$ 2E+11 $ & $199999695484.246439 $ & $ 453289.609568 $\\
$ 3E+11 $ & $299999423179.995211 $ & $ 554528.646163 $\\
$ 4E+11 $ & $399999101196.308601 $ & $ 640000.361434 $\\
$ 5E+11 $ & $499999105742.583455 $ & $ 715211.001138 $\\
$ 6E+11 $ & $599999250571.436655 $ & $ 783167.715577 $\\
$ 7E+11 $ & $699998999499.845475 $ & $ 845911.916175 $\\
$ 8E+11 $ & $799999133776.084743 $ & $ 904203.190001 $\\
$ 9E+11 $ & $899998818628.952024 $ & $ 958602.924046 $\\
$ 1E+12 $ & $999999030333.096225 $ & $ 1009803.669232 $\\
$ 2E+12 $ & $1999998755521.470649 $ & $ 1427105.865316 $\\
$ 3E+12 $ & $2999997819758.987859 $ & $ 1746299.820370 $\\
$ 4E+12 $ & $3999998370195.717561 $ & $ 2016279.693623 $\\
$ 5E+12 $ & $4999998073643.711478 $ & $ 2253672.042145 $\\
$ 6E+12 $ & $5999997276726.877147 $ & $ 2467566.593710 $\\
$ 7E+12 $ & $6999996936360.165729 $ & $ 2665065.541181 $\\
$ 8E+12 $ & $7999997864671.383505 $ & $ 2848858.049155 $\\
$ 9E+12 $ & $8999996425300.244577 $ & $ 3021079.319393 $\\
$ 1E+13 $ & $9999996988293.034200 $ & $ 3183704.089025 $\\
$ 2E+13 $ & $19999995126082.228688 $ & $ 4499685.436490 $\\
$ 3E+13 $ & $29999995531389.845427 $ & $ 5509328.368277 $\\
$ 4E+13 $ & $39999993533724.316829 $ & $ 6359550.652121 $\\
$ 5E+13 $ & $49999992543194.263655 $ & $ 7109130.001413 $\\
$ 6E+13 $ & $59999990297033.626198 $ & $ 7785491.725387 $\\
$ 7E+13 $ & $69999994316409.871731 $ & $ 8407960.376833 $\\
$ 8E+13 $ & $79999990160858.304239 $ & $ 8988688.375101 $\\
$ 9E+13 $ & $89999989501395.073897 $ & $ 9531798.550749 $\\
$ 1E+14 $ & $99999990573246.978538 $ & $ 10045400.569463 $\\
$ 2E+14 $ & $199999983475767.543204 $ & $ 14201359.711421 $\\
$ 3E+14 $ & $299999986702246.281944 $ & $ 17388356.540338 $\\
$ 4E+14 $ & $399999982296901.085038 $ & $ 20074942.600622 $\\
$ 5E+14 $ & $499999974019856.236519 $ & $ 22439658.012185 $\\
$ 6E+14 $ & $599999983610646.997632 $ & $ 24580138.242324 $\\
$ 7E+14 $ & $699999971887332.157455 $ & $ 26545816.027402 $\\
$ 8E+14 $ & $799999964680836.091645 $ & $ 28378339.693784 $\\
$ 9E+14 $ & $899999961386694.231242 $ & $ 30098146.961102 $\\
$ 1E+15 $ & $999999965752660.939840$ &   $31724269.567843$\\
\hline
\end{tabular}
\end{center}
\end{table}

\begin{table}[!hp]
\caption{Values of $\epsilon(x)$ for $\psi$ and $\theta$}
\label{table_eps}
\normalsize
$$\begin{array}{||l|l|l||l|l|l||}
\hline
b&\epsilon_\psi&\epsilon_\theta & b &\epsilon_\psi&\epsilon_\theta\\
\hline
20 & 6.123 E-4 & 6.606 E-4 & 100 & 2.903 E-11 & 2.903 E-11 \\
21 & 4.072 E-4 & 4.363 E-4 & 200 & 2.838 E-11 & 2.838 E-11 \\
22 & 2.706 E-4 & 2.881 E-4 & 300 & 2.772 E-11 & 2.772 E-11 \\
23 & 1.792 E-4 & 1.897 E-4 & 400 & 2.706 E-11 & 2.706 E-11 \\
24 & 1.183 E-4 & 1.247 E-4 & 500 & 2.641 E-11 & 2.641 E-11 \\
25 & 7.789 E-5 & 8.172 E-5 & 600 & 2.575 E-11 & 2.575 E-11 \\
\ln(10^{11})& 6.788 E-5 & 7.112 E-5 & 1000 & 2.315 E-11 & 2.315 E-11 \\
26 & 5.121 E-5 & 5.352 E-5 & 1250 & 2.153 E-11 & 2.153 E-11 \\
27 & 3.368 E-5 & 3.507 E-5 & 1500 & 1.991 E-11 & 1.991 E-11 \\
28 & 2.224 E-5 & 2.308 E-5 & 2000 & 1.671 E-11 & 1.671 E-11 \\
29 & 1.451 E-5 & 1.502 E-5 & 2200 & 1.544 E-11 & 1.544 E-11 \\
30 & 9.414 E-6 & 9.724 E-6 & 2500 & 1.355 E-11 & 1.355 E-11 \\
31 & 6.099 E-6 & 6.287 E-6 & 2800 & 1.169 E-11 & 1.169 E-11 \\
32 & 3.944 E-6 & 4.057 E-6 & 3000 & 1.047 E-11 & 1.047 E-11 \\
33 & 2.545 E-6 & 2.614 E-6 & 3200 & 9.267 E-12 & 9.267 E-12 \\
34 & 1.640 E-6 & 1.682 E-6 & 3300 & 8.658 E-12 & 8.658 E-12 \\
\ln(10^{15}) & 1.293 E-6 & 1.325 E-6 & 3400 & 8.083 E-12 & 8.083 E-12 \\
35 & 1.055 E-6 & 1.080 E-6 & 3455 & 7.750 E-12 & 7.750 E-12 \\
36 & 6.775 E-7 & 6.928 E-7 & 3500 & 7.488 E-12 & 7.488 E-12 \\
37 & 4.348 E-7 & 4.441 E-7 & 3600 & 6.930 E-12 & 6.930 E-12 \\
38 & 2.793 E-7 & 2.849 E-7 & 3700 & 6.351 E-12 & 6.351 E-12 \\
39 & 1.805 E-7 & 1.839 E-7 & 3750 & 6.080 E-12 & 6.080 E-12 \\
40 & 1.163 E-7 & 1.184 E-7 & 3800 & 5.821 E-12 & 5.821 E-12 \\
41 & 7.414 E-8 & 7.539 E-8 & 3850 & 5.533 E-12 & 5.533 E-12 \\
42 & 4.723 E-8 & 4.799 E-8 & 3900 & 5.259 E-12 & 5.259 E-12 \\
43 & 3.011 E-8 & 3.057 E-8 & 3950 & 4.999 E-12 & 4.999 E-12 \\
44 & 1.932 E-8 & 1.960 E-8 & 4000 & 4.751 E-12 & 4.751 E-12 \\
45 & 1.234 E-8 & 1.251 E-8 & 4050 & 4.496 E-12 & 4.496 E-12 \\
46 & 7.839 E-9 & 7.941 E-9 & 4100 & 4.231 E-12 & 4.231 E-12 \\
47 & 5.026 E-9 & 5.088 E-9 & 4150 & 3.981 E-12 & 3.981 E-12 \\
48 & 3.190 E-9 & 3.228 E-9 & 4200 & 3.746 E-12 & 3.746 E-12 \\
49 & 2.038 E-9 & 2.061 E-9 & 4300 & 3.308 E-12 & 3.308 E-12 \\
50 & 1.301 E-9 & 1.315 E-9 & 4400 & 2.844 E-12 & 2.844 E-12 \\
55 & 1.481 E-10 & 1.492 E-10 & 4500 & 2.445 E-12 & 2.445 E-12 \\
60 & 3.917 E-11 & 3.926 E-11 & 4700 & 1.774 E-12 & 1.774 E-12 \\
70 & 2.929 E-11 & 2.929 E-11 & 5000 & 9.562 E-13 & 9.562 E-13 \\
75 & 2.920 E-11 & 2.920 E-11 & 10000 & 6.341 E-18 & 6.341 E-18 \\
\hline
\end{array}$$
\end{table}

\begin{table}[!hp]
\caption{ Values of $\eta_k$ valid for $\exp(b_i)\leqs x\leqs \exp(b_{i+1})$.}
\protect{\label{table_eta1}}
\normalsize
$$\begin{array}{||r|r|r|r|r||}
\hline
b_i&\eta_1&\eta_2&\eta_3&\eta_4\\
\hline
20 & 1.388 E-2 & 2.914 E-1 & 6.118 E+0 & 1.285 E+2 \\
21 & 9.597 E-3 & 2.112 E-1 & 4.645 E+0 & 1.022 E+2 \\
22 & 6.625 E-3 & 1.524 E-1 & 3.505 E+0 & 8.061 E+1 \\
23 & 4.553 E-3 & 1.093 E-1 & 2.623 E+0 & 6.294 E+1 \\
24 & 3.116 E-3 & 7.790 E-2 & 1.948 E+0 & 4.869 E+1 \\
25 & 2.070 E-3 & 5.243 E-2 & 1.328 E+0 & 3.364 E+1 \\
\ln(10^{11}) & 1.849 E-3 & 4.808 E-2 & 1.250 E+0 & 3.250 E+1 \\
26 & 1.445 E-3 & 3.902 E-2 & 1.054 E+0 & 2.845 E+1 \\
27 & 9.820 E-4 & 2.750 E-2 & 7.699 E-1 & 2.156 E+1 \\
28 & 6.693 E-4 & 1.941 E-2 & 5.629 E-1 & 1.633 E+1 \\
29 & 4.504 E-4 & 1.352 E-2 & 4.054 E-1 & 1.216 E+1 \\
30 & 3.015 E-4 & 9.344 E-3 & 2.897 E-1 & 8.980 E+0 \\
31 & 2.012 E-4 & 6.437 E-3 & 2.060 E-1 & 6.592 E+0 \\
32 & 1.339 E-4 & 4.418 E-3 & 1.458 E-1 & 4.811 E+0 \\
33 & 8.887 E-5 & 3.022 E-3 & 1.028 E-1 & 3.493 E+0 \\
34 & 5.807 E-5 & 2.006 E-3 & 6.928 E-2 & 2.393 E+0 \\
\ln(10^{15})  & 4.637 E-5 & 1.623 E-3 & 5.680 E-2 & 1.988 E+0 \\
35 & 3.888 E-5 & 1.400 E-3 & 5.039 E-2 & 1.814 E+0 \\
36 & 2.564 E-5 & 9.484 E-4 & 3.509 E-2 & 1.299 E+0 \\
37 & 1.688 E-5 & 6.412 E-4 & 2.437 E-2 & 9.259 E-1 \\
38 & 1.112 E-5 & 4.333 E-4 & 1.690 E-2 & 6.591 E-1 \\
39 & 7.354 E-6 & 2.942 E-4 & 1.177 E-2 & 4.707 E-1 \\
40 & 4.853 E-6 & 1.990 E-4 & 8.157 E-3 & 3.345 E-1 \\
41 & 3.167 E-6 & 1.330 E-4 & 5.586 E-3 & 2.346 E-1 \\
42 & 2.064 E-6 & 8.872 E-5 & 3.815 E-3 & 1.641 E-1 \\
43 & 1.345 E-6 & 5.918 E-5 & 2.604 E-3 & 1.146 E-1 \\
44 & 8.818 E-7 & 3.968 E-5 & 1.786 E-3 & 8.036 E-2 \\
45 & 5.752 E-7 & 2.646 E-5 & 1.218 E-3 & 5.599 E-2 \\
46 & 3.733 E-7 & 1.755 E-5 & 8.245 E-4 & 3.875 E-2 \\
47 & 2.442 E-7 & 1.173 E-5 & 5.627 E-4 & 2.701 E-2 \\
48 & 1.582 E-7 & 7.749 E-6 & 3.797 E-4 & 1.861 E-2 \\
49 & 1.031 E-7 & 5.151 E-6 & 2.576 E-4 & 1.288 E-2 \\
50 & 7.229 E-8 & 3.976 E-6 & 2.187 E-4 & 1.203 E-2 \\
55 & 8.952 E-9 & 5.371 E-7 & 3.223 E-5 & 1.934 E-3 \\
60 & 2.748 E-9 & 1.924 E-7 & 1.347 E-5 & 9.425 E-4 \\
70 & 2.197 E-9 & 1.648 E-7 & 1.236 E-5 & 9.268 E-4 \\
75 & 2.920 E-9 & 2.920 E-7 & 2.920 E-5 & 2.920 E-3 \\
\hline
\end{array}$$
\end{table}
\begin{table}[!hp]
\caption{Values of $\eta_k$ (continued)}
\label{table_eta2}
\normalsize
$$\begin{array}{||r|r|r|r|r||}
\hline
b_i&\eta_1&\eta_2&\eta_3&\eta_4\\
\hline
100 & 5.805 E-9 & 1.161 E-6 & 2.322 E-4 & 4.644 E-2 \\
200 & 8.512 E-9 & 2.554 E-6 & 7.661 E-4 & 2.299 E-1 \\
300 & 1.109 E-8 & 4.434 E-6 & 1.774 E-3 & 7.094 E-1 \\
400 & 1.353 E-8 & 6.765 E-6 & 3.383 E-3 & 1.692 E+0 \\
500 & 1.585 E-8 & 9.505 E-6 & 5.703 E-3 & 3.422 E+0 \\
600 & 2.575 E-8 & 2.575 E-5 & 2.575 E-2 & 2.575 E+1 \\
1000 & 2.893 E-8 & 3.616 E-5 & 4.520 E-2 & 5.650 E+1 \\
1250 & 3.229 E-8 & 4.843 E-5 & 7.265 E-2 & 1.090 E+2 \\
1500 & 3.982 E-8 & 7.963 E-5 & 1.593 E-1 & 3.185 E+2 \\
2000 & 3.675 E-8 & 8.084 E-5 & 1.779 E-1 & 3.913 E+2 \\
2200 & 3.859 E-8 & 9.646 E-5 & 2.412 E-1 & 6.029 E+2 \\
2500 & 3.794 E-8 & 1.063 E-4 & 2.975 E-1 & 8.328 E+2 \\
2800 & 3.507 E-8 & 1.053 E-4 & 3.157 E-1 & 9.469 E+2 \\
3000 & 3.351 E-8 & 1.073 E-4 & 3.431 E-1 & 1.098 E+3 \\
3200 & 3.058 E-8 & 1.010 E-4 & 3.331 E-1 & 1.099 E+3 \\
3300 & 2.944 E-8 & 1.001 E-4 & 3.403 E-1 & 1.157 E+3 \\
3400 & 2.793 E-8 & 9.648 E-5 & 3.334 E-1 & 1.152 E+3 \\
3455 & 2.713 E-8 & 9.494 E-5 & 3.323 E-1 & 1.163 E+3 \\
3500 & 2.696 E-8 & 9.704 E-5 & 3.494 E-1 & 1.258 E+3 \\
3600 & 2.565 E-8 & 9.488 E-5 & 3.511 E-1 & 1.299 E+3 \\
3700 & 2.382 E-8 & 8.931 E-5 & 3.350 E-1 & 1.256 E+3 \\
3750 & 2.311 E-8 & 8.780 E-5 & 3.337 E-1 & 1.268 E+3 \\
3800 & 2.241 E-8 & 8.628 E-5 & 3.322 E-1 & 1.279 E+3 \\
3850 & 2.158 E-8 & 8.416 E-5 & 3.282 E-1 & 1.280 E+3 \\
3900 & 2.078 E-8 & 8.205 E-5 & 3.241 E-1 & 1.281 E+3 \\
3950 & 2.000 E-8 & 7.997 E-5 & 3.199 E-1 & 1.280 E+3 \\
4000 & 1.924 E-8 & 7.793 E-5 & 3.156 E-1 & 1.279 E+3 \\
4050 & 1.844 E-8 & 7.557 E-5 & 3.099 E-1 & 1.271 E+3 \\
4100 & 1.756 E-8 & 7.286 E-5 & 3.024 E-1 & 1.255 E+3 \\
4150 & 1.672 E-8 & 7.022 E-5 & 2.950 E-1 & 1.239 E+3 \\
4200 & 1.611 E-8 & 6.927 E-5 & 2.979 E-1 & 1.281 E+3 \\
4300 & 1.456 E-8 & 6.404 E-5 & 2.818 E-1 & 1.240 E+3 \\
4400 & 1.280 E-8 & 5.759 E-5 & 2.592 E-1 & 1.167 E+3 \\
4500 & 1.150 E-8 & 5.401 E-5 & 2.539 E-1 & 1.194 E+3 \\
4700 & 8.868 E-9 & 4.434 E-5 & 2.217 E-1 & 1.109 E+3 \\
\hline
\end{array}$$
\end{table}
\begin{table}[!hp]
\label{table_012}
\caption{Values for $\theta(x)$}
\normalsize
$$
\begin{array}{||l||l|l||l|l||l|l||}
\hline
      & a_0     &  b_0    & a_1     &  b_1      & a_2     &  b_2     \\
\hline
1E+08 & 0.99985 & 0.99998 & 0.00275 & -0.00044 & 0.05062 & -0.00851 \\
2E+08 & 0.99989 & 0.99997 & 0.00201 & -0.00065 & 0.03847 & -0.01275 \\
3E+08 & 0.99991 & 0.99998 & 0.00165 & -0.00057 & 0.03256 & -0.01131 \\
4E+08 & 0.99993 & 0.99998 & 0.00124 & -0.00049 & 0.02447 & -0.00988 \\
5E+08 & 0.99992 & 0.99998 & 0.00152 & -0.00052 & 0.03039 & -0.01061 \\
6E+08 & 0.99994 & 0.99999 & 0.00103 & -0.00038 & 0.02095 & -0.00785 \\
7E+08 & 0.99993 & 0.99998 & 0.00126 & -0.00051 & 0.02577 & -0.01054 \\
8E+08 & 0.99994 & 0.99998 & 0.00110 & -0.00044 & 0.02268 & -0.00916 \\
9E+08 & 0.99994 & 0.99998 & 0.00117 & -0.00050 & 0.02398 & -0.01050 \\
1E+09 & 0.99995 & 0.99999 & 0.00096 & -0.00021 & 0.02009 & -0.00455 \\
2E+09 & 0.99996 & 1.00000 & 0.00067 & -0.00018 & 0.01427 & -0.00401 \\
3E+09 & 0.99997 & 1.00000 & 0.00062 & -0.00015 & 0.01340 & -0.00333 \\
4E+09 & 0.99997 & 1.00000 & 0.00050 & -0.00017 & 0.01087 & -0.00390 \\
5E+09 & 0.99997 & 1.00000 & 0.00046 & -0.00018 & 0.01026 & -0.00410 \\
6E+09 & 0.99998 & 1.00000 & 0.00040 & -0.00013 & 0.00900 & -0.00313 \\
7E+09 & 0.99998 & 1.00000 & 0.00045 & -0.00018 & 0.01011 & -0.00415 \\
8E+09 & 0.99998 & 1.00000 & 0.00034 & -0.00016 & 0.00774 & -0.00367 \\
9E+09 & 0.99998 & 1.00000 & 0.00037 & -0.00010 & 0.00840 & -0.00249 \\
1E+10 & 0.99998 & 1.00000 & 0.00038 & -0.00008 & 0.00876 & -0.00203 \\
2E+10 & 0.99998 & 1.00000 & 0.00025 & -0.00006 & 0.00584 & -0.00152 \\
3E+10 & 0.99999 & 1.00000 & 0.00020 & -0.00005 & 0.00473 & -0.00122 \\
4E+10 & 0.99999 & 1.00000 & 0.00018 & -0.00007 & 0.00431 & -0.00183 \\
5E+10 & 0.99999 & 1.00000 & 0.00019 & -0.00004 & 0.00458 & -0.00119 \\
6E+10 & 0.99999 & 1.00000 & 0.00015 & -0.00006 & 0.00356 & -0.00167 \\
7E+10 & 0.99999 & 1.00000 & 0.00014 & -0.00004 & 0.00338 & -0.00117 \\
8E+10 & 0.99999 & 1.00000 & 0.00011 & -0.00006 & 0.00276 & -0.00164 \\
9E+10 & 0.99999 & 1.00000 & 0.00011 & -0.00004 & 0.00262 & -0.00114 \\
1E+11 & 0.99999 & 1.00000 & 0.00014 & -0.00002 & 0.00341 & -0.00058 \\
2E+11 & 0.99999 & 1.00000 & 0.00008 & -0.00002 & 0.00206 & -0.00057 \\
3E+11 & 0.99999 & 1.00000 & 0.00007 & -0.00001 & 0.00170 & -0.00031 \\
4E+11 & 0.99999 & 1.00000 & 0.00008 & -0.00002 & 0.00188 & -0.00061 \\
5E+11 & 0.99999 & 1.00000 & 0.00006 & -0.00001 & 0.00138 & -0.00038 \\
6E+11 & 0.99999 & 1.00000 & 0.00005 & -0.00002 & 0.00131 & -0.00070 \\
7E+11 & 0.99999	& 1.00000	& 0.00006	& -0.00001 & 0.00144 & -0.00039 \\
\hline
\end{array}
$$
where the constants satisfy for  $n\cdot10^k\leqs x\leqs (n+1)\cdot10^k$
$$
\begin{array}{rcl}
 a_0 x&\leqs \theta(x)\leqs& b_0 x\\
x-a_1 \frac{x}{\ln x}&\leqs\theta(x)\leqs& x+b_1 \frac{x}{\ln x}\\
x-a_2\frac{x}{\ln^2 x}&\leqs\theta(x) \leqs& x+b_2\frac{x}{\ln^2 x}
\end{array}
$$
up to $8\cdot 10^{11}$.
\end{table}
\begin{table}[!hp]
\label{table_834}
\caption{Values for $p_k$ and $\theta(p_k)$}
\normalsize
$$
\begin{array}{||l||l|l||l|l||l|l||}
\hline
      & a_8     &  b_8    & a_3     &  b_3      & a_4     &  b_4     \\
\hline
1E+08 & 2.07947 & 2.07516 & 0.95665 & 0.95433 & 2.07207 & 2.03783 \\
2E+08 & 2.07517 & 2.07280 & 0.95517 & 0.95405 & 2.06330 & 2.04341 \\
3E+08 & 2.07281 & 2.07122 & 0.95493 & 0.95379 & 2.06236 & 2.04535 \\
4E+08 & 2.07123 & 2.07005 & 0.95459 & 0.95403 & 2.06210 & 2.05123 \\
5E+08 & 2.07006 & 2.06909 & 0.95448 & 0.95357 & 2.06053 & 2.04534 \\
6E+08 & 2.06910 & 2.06833 & 0.95448 & 0.95374 & 2.06271 & 2.05150 \\
7E+08 & 2.06834 & 2.06767 & 0.95421 & 0.95342 & 2.05976 & 2.04701 \\
8E+08 & 2.06768 & 2.06710 & 0.95411 & 0.95342 & 2.05999 & 2.04871 \\
9E+08 & 2.06711 & 2.06660 & 0.95395 & 0.95336 & 2.05808 & 2.04762 \\
1E+09 & 2.06661 & 2.06350 & 0.95409 & 0.95319 & 2.06243 & 2.04925 \\
2E+09 & 2.06351 & 2.06183 & 0.95355 & 0.95311 & 2.05913 & 2.05126 \\
3E+09 & 2.06184 & 2.06070 & 0.95336 & 0.95297 & 2.05821 & 2.05036 \\
4E+09 & 2.06071 & 2.05985 & 0.95322 & 0.95295 & 2.05684 & 2.05159 \\
5E+09 & 2.05986 & 2.05917 & 0.95314 & 0.95288 & 2.05600 & 2.05103 \\
6E+09 & 2.05918 & 2.05862 & 0.95313 & 0.95287 & 2.05643 & 2.05149 \\
7E+09 & 2.05863 & 2.05815 & 0.95305 & 0.95276 & 2.05517 & 2.04994 \\
8E+09 & 2.05816 & 2.05774 & 0.95300 & 0.95283 & 2.05489 & 2.05168 \\
9E+09 & 2.05775 & 2.05739 & 0.95300 & 0.95277 & 2.05540 & 2.05073 \\
1E+10 & 2.05740 & 2.05515 & 0.95301 & 0.95264 & 2.05571 & 2.04968 \\
2E+10 & 2.05516 & 2.05395 & 0.95283 & 0.95263 & 2.05364 & 2.04983 \\
3E+10 & 2.05396 & 2.05313 & 0.95275 & 0.95262 & 2.05225 & 2.04964 \\
4E+10 & 2.05314 & 2.05252 & 0.95272 & 0.95260 & 2.05143 & 2.04899 \\
5E+10 & 2.05253 & 2.05203 & 0.95272 & 0.95258 & 2.05127 & 2.04869 \\
6E+10 & 2.05204 & 2.05163 & 0.95269 & 0.95260 & 2.05060 & 2.04875 \\
7E+10 & 2.05164 & 2.05129 & 0.95270 & 0.95260 & 2.05060 & 2.04865 \\
8E+10 & 2.05130 & 2.05099 & 0.95267 & 0.95262 & 2.04978 & 2.04877 \\
9E+10 & 2.05100 & 2.05073 & 0.95269 & 0.95262 & 2.04982 & 2.04879 \\
1E+11	& 2.05074	& 2.04910	& 0.95271	& 0.95259	& 2.05014	& 2.04734 \\
2E+11	& 2.04911	& 2.04821	& 0.95273	& 0.95266	& 2.04851	& 2.04668 \\
3E+11	& 2.04822	& 2.04761	& 0.95276	& 0.95270	& 2.04775	& 2.04621 \\
4E+11	& 2.04762	& 2.04716	& 0.95278	& 0.95271	& 2.04692	& 2.04601 \\
5E+11	& 2.04717	& 2.04680	& 0.95282	& 0.95276	& 2.04670	& 2.04579 \\
6E+11	& 2.04681	& 2.04651	& 0.95283	& 0.95279	& 2.04617	& 2.04543 \\
7E+11 & 2.04652	& 2.04625	& 0.95286	& 0.95280	& 2.04597	& 2.04524 \\
\hline
\end{array}$$

where the constants satisfy for  $n\cdot10^m\leqs p_k\leqs (n+1)\cdot10^m$

$$\begin{array}{rcl}
k\left(\ln k+\ln_2 k-1+\frac{\ln_2 k-a_8}{\ln k}\right)&\leqs \theta(p_k)
\leqs & k\left(\ln k+\ln_2 k-1+\frac{\ln_2 k-b_8}{\ln k}\right)\\
k\left(\ln k+\ln_2 k-a_3\right)&\leqs p_k\leqs& k\left(\ln k+\ln_2 k-b_3\right)\\
k\left(\ln k+\ln_2 k-1+\frac{\ln_2 k-a_4}{\ln k}\right)&\leqs p_k\leqs& k\left(\ln k+\ln_2 k-1+\frac{\ln_2 k-b_4}{\ln k}\right)\\
\end{array}$$
up to $8\cdot 10^{11}$.
\end{table}
\begin{table}[!hp]
\label{table_567}
\caption{Values for $\pi(x)$}
\normalsize
$$
\begin{array}{||l||l|l||l|l||l|l||}
\hline
      & a_5     &  b_5    & a_6     &  b_6      & a_7     &  b_7     \\
\hline
1E+08 & 1.12379 & 1.13015 & 2.36474 & 2.40986 & 1.06139 & 1.06514\\
2E+08 & 1.12113 & 1.12429 & 2.35944 & 2.38213 & 1.06022 & 1.06184\\
3E+08 & 1.11922 & 1.12158 & 2.35351 & 2.37715 & 1.05917 & 1.06074\\
4E+08 & 1.11802 & 1.11954 & 2.35815 & 2.36977 & 1.05878 & 1.05964\\
5E+08 & 1.11642 & 1.11818 & 2.34796 & 2.36727 & 1.05793 & 1.05906\\
6E+08 & 1.11556 & 1.11706 & 2.35087 & 2.36678 & 1.05756 & 1.05858\\
7E+08 & 1.11455 & 1.11587 & 2.34345 & 2.36058 & 1.05696 & 1.05793\\
8E+08 & 1.11363 & 1.11492 & 2.34295 & 2.35791 & 1.05656 & 1.05745\\
9E+08 & 1.11327 & 1.11405 & 2.34153 & 2.35399 & 1.05647 & 1.05700\\
1E+09 & 1.10903 & 1.11346 & 2.33374 & 2.35650 & 1.05441 & 1.05673\\
2E+09 & 1.10679 & 1.10928 & 2.33043 & 2.34118 & 1.05336 & 1.05467\\
3E+09 & 1.10527 & 1.10683 & 2.32501 & 2.33314 & 1.05265 & 1.05342\\
4E+09 & 1.10395 & 1.10535 & 2.32169 & 2.32929 & 1.05195 & 1.05272\\
5E+09 & 1.10306 & 1.10408 & 2.31900 & 2.32487 & 1.05152 & 1.05208\\
6E+09 & 1.10226 & 1.10314 & 2.31746 & 2.32381 & 1.05114 & 1.05162\\
7E+09 & 1.10150 & 1.10234 & 2.31325 & 2.32076 & 1.05074 & 1.05124\\
8E+09 & 1.10096 & 1.10161 & 2.31378 & 2.31816 & 1.05049 & 1.05084\\
9E+09 & 1.10054 & 1.10103 & 2.31191 & 2.31673 & 1.05034 & 1.05057\\
1E+10 & 1.09706 & 1.10062 & 2.30164 & 2.31683 & 1.04856 & 1.05041\\
2E+10 & 1.09520 & 1.09708 & 2.29665 & 2.30469 & 1.04764 & 1.04858\\
3E+10 & 1.09396 & 1.09524 & 2.29308 & 2.29816 & 1.04703 & 1.04767\\
4E+10 & 1.09295 & 1.09398 & 2.28955 & 2.29441 & 1.04651 & 1.04706\\
5E+10 & 1.09220 & 1.09298 & 2.28817 & 2.29108 & 1.04616 & 1.04655\\
6E+10 & 1.09157 & 1.09223 & 2.28576 & 2.28917 & 1.04585 & 1.04619\\
7E+10 & 1.09100 & 1.09158 & 2.28458 & 2.28741 & 1.04556 & 1.04587\\
8E+10 & 1.09050 & 1.09102 & 2.28282 & 2.28500 & 1.04531 & 1.04558\\
9E+10 & 1.09010 & 1.09052 & 2.28203 & 2.28347 & 1.04511 & 1.04532\\
1E+11	& 1.08738	& 1.09012	& 2.27312	& 2.28311	& 1.04375	& 1.04514\\
2E+11	& 1.08583	& 1.08739	& 2.26810	& 2.27413	& 1.04297	& 1.04377\\
3E+11	& 1.08477	& 1.08585	& 2.26471	& 2.26852	& 1.04244	& 1.04299\\
4E+11	& 1.08398	& 1.08478	& 2.26238	& 2.26507	& 1.04205	& 1.04245\\
5E+11	& 1.08335	& 1.08399	& 2.26052	& 2.26261	& 1.04174	& 1.04206\\
6E+11	& 1.08281	& 1.08337	& 2.25875	& 2.26085	& 1.04147	& 1.04175\\
7E+11 & 1.08236	& 1.08282	& 2.25747	& 2.25901	& 1.04124	& 1.04148\\
\hline
\end{array}$$

where the constants satisfy for  $n\cdot10^k\leqs x\leqs (n+1)\cdot10^k$

$$\begin{array}{rcl}
\frac{x}{\ln x}\left(1+\frac{a_5}{\ln x}\right)&<
\pi(x)\leqs& \frac{x}{\ln x}\left(1+\frac{b_5}{\ln x}\right)\\
\frac{x}{\ln x}\left(1+\frac{1}{\ln x}+\frac{a_6}{\ln^2 x}\right)&\leqs
\pi(x)\leqs&\frac{x}{\ln x}\left(1+\frac{1}{\ln x}+\frac{b_6}{\ln^2 x}\right)\\
\frac{x}{\ln x-a_7}&\leqs
\pi(x)\leqs&\frac{x}{\ln x-b_7}
\end{array}$$
up to $8\cdot 10^{11}$.
\end{table}

\end{document}